\documentclass[12pt,a4paper,twoside]{article}

\usepackage[centertags]{amsmath}
\usepackage{amsfonts}
\usepackage{amssymb}
\usepackage{amsthm}
\usepackage{ulem}
\usepackage{epsfig}
\usepackage{subfigure}
\usepackage{color}
\usepackage{mathptmx}
\usepackage{mathtools}
\usepackage{mathrsfs}
\usepackage{latexsym}
\usepackage{array}
\usepackage{titlesec}
\usepackage{graphicx}
\usepackage{verbatim}

\newtheoremstyle{break}
  {\topsep}{\topsep}%
  {\itshape}{}%
  {\bfseries}{}%
  {\newline}{}%
\newtheoremstyle{break1}
  {\topsep}{\topsep}%
  {}{}%
  {\bfseries}{}%
  {\newline}{}%
\theoremstyle{break} \newtheorem{theorem}{Theorem}[section]
\theoremstyle{break} 
\theoremstyle{break1} \newtheorem{remark}[theorem]{Remark}
\theoremstyle{break} 
\theoremstyle{break1}  
\theoremstyle{break} \newtheorem{lemma}[theorem]{Lemma}
\theoremstyle{break} \newtheorem{corollary}[theorem]{Corollary}
\theoremstyle{break} \newtheorem{example}[theorem]{Example}
\theoremstyle{break} 
\theoremstyle{break1} \newtheorem{problem}[theorem]{Problem}
\theoremstyle{break} 
\theoremstyle{break} 
\theoremstyle{break} 
\numberwithin{equation}{section}

\newcommand{\hide}[1]{}
\newcommand{\N}{{\mathbb{N}}}
\newcommand{\R}{{\mathbb{R}}}
\newcommand{\D}{{\mathbb{D}}}
\newcommand{\C}{{\mathbb{C}}}

\def\MD{\mathop{\mathcal{M}(\D)}
}

\begin{document}

\vspace*{0.7cm}

\renewcommand{\thefootnote}{}
\begin{center}
{\Large \bf A Schwarz lemma for locally univalent\\[2mm] meromorphic functions}\\[5mm]

Richard Fournier, Daniela Kraus and Oliver Roth\\[5mm]
\textit{In Memory of Stephan Ruscheweyh}
\end{center}

\footnotetext{Mathematics Subject Classification (2000) \quad Primary 30C55}

\section{Introduction and main results}

Let  $\MD$ denote the set of all meromorphic functions in the unit disk
$\D$ of the complex plane $\C$.
 \textit{Marty's fundamental normality criterion} \cite{Ma}, see also
  \cite[\S 3.3]{Sc}, says that a family
$\mathcal{F} \subseteq \MD$ is normal if and only if the family of spherical derivatives
$$ f^{\sharp}(z):=\frac{|f'(z)|}{1+|f(z)|^2}$$
of all $f \in \mathcal{F}$ is locally bounded (above) in $\D$.
 Some years ago, J.~Grahl and
S.~Nevo \cite{GN} proved the surprising result that for any $c>0$ the family
$$\mathcal{F}_c:=\left\{f \in \MD \, : \, f^{\sharp}(z) \ge c \text{ for all } 
  z\in  \D\right\} \, ,$$
consisting of all meromorphic functions in $\D$  with spherical derivative  uniformly
bounded from \textit{below},  is also normal.

\medskip

The original proof in \cite{GN} is fairly involved and is based on a sophisticated application
of Zalcman's lemma \cite{Za}. Combining the result of Grahl and Nevo with Marty's
criterion shows that a \textit{uniform} lower bound for $f^{\sharp}$ has to result in a
\textit{locally uniform} upper bound for $f^{\sharp}$. In fact, such an upper bound has been
given by Steinmetz \cite{St}, who proved that

\begin{equation} \label{eq:ba1} 
 f^{\sharp}(z) \le \frac{1}{c
  \left(1-|z|^2\right)^2} \, , \qquad z \in \D, \, f \in \mathcal{F}_c  \, . 
\end{equation}
The approach in \cite{St} is based on the theory of complex differential equations
and leads in particular to  a short proof of the result of Grahl and Nevo.

\medskip

The main purpose of this paper is to prove the following Schwarz--type lemma for
functions in the classes $\mathcal{F}_c$, which provides in particular a \textit{sharp} form of 
inequality (\ref{eq:ba1}) for the point $z=0$ including a precise description of the
extremal functions.

\begin{theorem}[Schwarz lemma for $\mathcal{F}_c$] \label{cor:2}
  Let $c>0$ and  $f \in \mathcal{F}_c$. Then the following hold.
\begin{itemize}
\item[(a)] $c \le 1/2$.

\item[(b)] $f^{\sharp}(0) \le \displaystyle \frac{1+\sqrt{1-4 c^2}}{2c}$.\\
\item[(c)] $f^{\sharp}(0) \ge \displaystyle \frac{1-\sqrt{1-4 c^2}}{2c}$.

\end{itemize}
Equality holds in either case if and only if $f(z)=T(\eta z)$ where $T$ is a
rigid motion of the Riemann sphere and $\eta \in \C$ such that
$$ |\eta|=1 \text{ in } (a)\, , \quad |\eta|=\frac{1+\sqrt{1-4 c^2}}{2 c}
\text{ in } (b) \, , \quad |\eta|=\frac{1-\sqrt{1-4 c^2}}{2 c} \text{ in } (c)
\, .$$
In particular, $\mathcal{F}_{1/2}$ is precisely the set of  rigid motions
of the Riemann sphere.
\end{theorem}

The  proof of Theorem \ref{cor:2} is deceptively simple and only uses the minimum principle for
superharmonic functions. In addition, Theorem \ref{cor:2}  immediately
yields not only the normality criterion of Grahl and Nevo, but our method of
 proof also
leads to a sharpening of the quantitative upper bound (\ref{eq:ba1})  for
$f^{\sharp}$ as well as a corresponding \textit{lower} bound:

\begin{theorem} \label{thm:z0}
Suppose that $f \in \mathcal{F}_c$ for
some $c>0$. Then
\begin{equation} \label{eq:5}
\frac{1-\sqrt{1-4 c^2 \left(1-|z_0|^2 \right)^2}}{2c \left(1-|z_0|^2
  \right)^2}  \le f^{\sharp}(z_0) \le \frac{1+\sqrt{1-4 c^2 \left(1-|z_0|^2 \right)^2}}{2 c
    \left(1-|z_0|^2 \right)^2} \, , \quad z_0 \in \D \, . 
\end{equation}
Both estimates are sharp if and only if $z_0=0$.
\end{theorem}

Even though the right--hand inequality in (\ref{eq:5}) clearly improves (\ref{eq:ba1}),
 both inequalities lead to the same asymptotic estimate, namely
\begin{equation} \label{eq:5a} \limsup \limits_{|z| \to 1} \left(1-|z|^2 \right)^2 \,
  f^{\sharp}(z)  \le \frac{1}{c} \, , \qquad f \in \mathcal{F}_c \, . 
\end{equation}
Recent work of Gr\"ohn \cite[Theorem 3]{Janne} shows that there is a function $f
\in \bigcup_{c>0}\mathcal{F}_c$ such that
$$
\inf \limits_{n \in \N}  \left(1-|z_n|^2 \right)^2 \, f^{\sharp}(z_n)>0
\,
$$
for some sequence $(z_n)$ in $\D$ with $|z_n| \to 1$. Hence, for sufficiently small values of $c>0$ inequality (\ref{eq:5a}) is
sharp up to a multiplicative constant. The following result provides an
improved upper bound for this constant. In fact it shows that for all possible
values of $c$ one can replace the
number $1$ on the right--hand side of (\ref{eq:5a}) by $(3-\sqrt{5})/2 \approx 0.38$:

\begin{theorem} \label{thm:z01}
Let $c \in (0,1/2)$. Then for any $f \in \mathcal{F}_c$, 
\begin{equation} \label{eq:51}
f^{\sharp}(z_0)  \le \left( \frac{\sqrt{4+|z_0|^2}-|z_0|}{2} \right)^2 
    \frac{1}{c \left(1-|z_0|^2 \right)^2} \, , \qquad z_0 \in \D \, . 
  \end{equation}
\end{theorem}

The proof of Theorem \ref{thm:z01} is different from the proof of Theorem
\ref{thm:z0} and is based on complex differential equations and a refinement of
the Schwarz--Pick lemma for bounded holomorphic functions (Lemma \ref{lem:sp}).

\medskip

We return to the Schwarz
lemma for $\mathcal{F}_c$ (Theorem \ref{cor:2}) and briefly indicate the second main goal of this note:
We shall   show 
 that the ``interior'' extremal problem solved by Theorem \ref{cor:2}
 is related to a ``dual'' nonlinear boundary value
problem and thereby to Beurling's extension \cite{B} of the Riemann mapping
theorem, see Section 3 below. In
particular, we establish the following rigidity property of
 locally conformal maps with ``length--preserving boundary distortion'':

\begin{theorem} \label{thm:kuh}
  Let $c>0$ and $f \in \mathcal{M}(\D)$ be locally univalent such that $f(0)=0$ and
  $$ \lim \limits_{z \to \xi} \frac{|f'(z)|}{1+|f(z)|^2}=c \quad \text{ for
    all } \, |\xi|=1 \, .$$
  Then $f(z)=\eta z$ for some $\eta \in \C$.
  \end{theorem}

This answers the question raised by K\"uhnau in \cite{K} (Remark after Satz
2') who proved  Theorem \ref{thm:kuh}
 under the additional condition that $f$ is holomorphic on a neighborhood of
$\overline{\D}$ and $f(\D) \subseteq \D$, and asked if this condition is
necessary. Now geometrically,  it is natural to think of a function $f \in \mathcal{M}(\D)$
as a map $f : (\D,d_{\D}) \to (\hat{\C},d_{\hat{\C}})$ from
the unit disk $\D$ equipped with the hyperbolic distance $d_{\D}$  into the
Riemann sphere $\hat{\C}$
equipped with the spherical distance $d_{\hat{\C}}$. We call such a map
$f :(\D,d_{\D}) \to (\hat{\C},d_{\hat{\C}})$ length-preserving on the circle
$|z|=\varrho<1$ if
for each subarc $\gamma$ of $|z|=\varrho$ the
spherical length of $f(\gamma)$ is exactly the hyperbolic length of
$\gamma$.
Theorem \ref{thm:kuh} then
 has the following, perhaps appealing
geometric interpretation, compare \cite{K} for the special case  $f(\D) \subseteq \D$.

\begin{corollary} \label{cor:4}
Let $f \in \mathcal{M}(\D)$ be a locally univalent function such that $f(0)=0$ and
$0<\varrho<1$. Suppose that $f :(\D,d_{\D}) \to (\hat{\C},d_{\hat{\C}})$ is
length-preserving on the circle $|z|=\varrho$. Then $f(z)=\eta z$ for some
$\eta \in \C$.
\end{corollary}

In fact, Theorem \ref{cor:2} shows that only $c \le 1/2$ is possible in
Theorem \ref{thm:kuh}. A short computation (compare \cite{K}) implies that
this means that only $\varrho \le \sqrt{2}-1$ is 
possible in Corollary \ref{cor:4}. Hence there are  locally univalent
meromorphic maps $f :
(\D,d_{\D}) \to (\hat{\C},d_{\hat{\C}})$, $f(0)=0$, which are length--preserving on a
circle $|z|=\varrho$ if and only if $\varrho \le \sqrt{2}-1$. In contrast,
there are ``hyperbolically'' length--preserving locally univalent maps $f :
(\D,d_{\D}) \to (\D,d_{\D})$ (analogously defined)  of the form $f(z)=\eta z$ on any circle $|z|=\varrho
\in (0,1)$.

\medskip

The plan of this note is as follows. In Section 2 we prove Theorems
\ref{cor:2}, \ref{thm:z0} and \ref{thm:z01}. In Section 3 we discuss the
relation of the Schwarz lemma for $\mathcal{F}_c$ with Beurling's extension of
the Riemann mapping theorem  and in particular prove Theorem
\ref{thm:kuh}. The main additional tool is
a celebrated result of Gidas, Ni and Nirenberg \cite{GNN}
concerning positive solutions of certain semilinear elliptic PDEs.
 The paper concludes with Section 4, which presents a simple direct proof of the
Grahl--Nevo normality criterion and a quantitative normality result for
rational functions in $\mathcal{F}_c$ based on a Bernstein--type
inequality for rational functions due to Borwein and Erdelyi \cite{BE}.

\section{Proof of Theorems \ref{cor:2}, \ref{thm:z0} and \ref{thm:z01}}

\begin{proof}[Proof of Theorems \ref{cor:2} and \ref{thm:z0}] Let $z_0 \in \D$.
 Since postcomposing $f \in \mathcal{F}_c$ with a rigid motion of the
Riemann sphere does not change the spherical derivative of $f$, we may assume that $f(z_0)=0$.
\medskip

(i) \, We consider the unit disk automorphism
 $$S(z)=\frac{z+z_0}{1+\overline{z_0}z } \, , $$
and  the auxiliary functions
$$ u(z):=\log f^{\sharp}(S(z))=\log \frac{|f'(S(z))|}{1+|f(S(z))|^2}$$
and
$$  v_r(z):=\log \frac{|f'(S(z))|}{1+\left| \displaystyle \frac{f(S(z))}{z} r
  \right|^2}\, , \quad 0 < r \le 1 \, .$$
Since $f$ is a locally univalent meromorphic function, the function $\log f^{\sharp}$ is
smooth and in fact satisfies Liouville's equation
$$ \Delta \log f^{\sharp}=-4 f^{\sharp}(z)^2  \quad \text{ in } \D , $$
as it is well--known and also easy to show by direct computation. Hence, $u$
is superharmonic on $\D$. Since $S(0)=z_0$ and $f(z_0)=0$, we see that also
$v_r$ is a  smooth function on $\D$. In view of
$$ v_r(z)=\log g_r^{\sharp}(z)+\log \left| \frac{ f'(S(z))}{ g_r'(z)} \right|
\, , $$
where
$$ g_r(z):=\frac{(f \circ S)(z)}{z} r \in \mathcal{M}(\D) \, , $$
we see that
\begin{equation} \label{eq:vr}
  \Delta v_r(z)=-4 g_r^{\sharp}(z)^2 \, , \qquad z \in\D \, .
  \end{equation}
Hence $v_r$ is a superharmonic function on $\D$ as well.

\medskip

(ii) \, 
Since for each $0<r<1$,
$$u =v_r \quad \text{ on } |z|=r \, , $$
the minimum principle applied to the  superharmonic functions $u$ and $v_r$  shows
$$ \inf \limits_{|z| \le r} u(z) =\inf \limits_{|z| \le r} v_r(z) \le v_r(0) \, .$$ 
Letting $r \to 1$, we get 
\begin{equation} \label{eq:6}
 \log c \le \inf \limits_{z \in \D} \log f^{\sharp}(z)   =\inf \limits_{z \in \D}
u(z)=\inf \limits_{z \in \D} v_1(z)  \le   v_1(0)\, ,
\end{equation}
or, equivalently,
$$ c \le \inf \limits_{z \in \D} f^{\sharp}(z)  \le
\frac{|f'(z_0)|}{1+\left(1-|z_0|^2 \right)^2 |f'(z_0)|^2} \, .$$
Hence
$$ c \left(1-|z_0|^2 \right)^2|f'(z_0)|^2-|f'(z_0)|+c \le 0 \, .$$
Now, the quadratic function $$\varrho(x):=c (1-|z_0|^2)^2 x^2-x+c$$ has the
zeros
$$ \frac{1\pm \sqrt{1-4 c^2 \left(1-|z_0|^2 \right)^2}}{2c \left(1-|z_0|^2
  \right)^2} \, , $$
so it takes on nonpositive values if and only if
$$ c \le \frac{1}{2 (1-|z_0|^2)} \, .$$
In particular, this shows that $c \le 1/2$ and
\begin{equation} \label{eq:7}
\frac{1-\sqrt{1-4 c^2 \left(1-|z_0|^2 \right)^2}}{2c \left(1-|z_0|^2 \right)^2} \le |f'(z_0)| \le \frac{1+\sqrt{1-4 c^2\left(1-|z_0|^2 \right)^2}}{2c\left(1-|z_0|^2 \right)^2}\, .
\end{equation}
If equality holds on either side of (\ref{eq:7}), then equality holds in both inequalities of (\ref{eq:6}).
Therefore the minimum principle shows that $v_1$ is constant. In view of
(\ref{eq:vr}), we get
 that $g_1$ is constant and hence
$f(S(z))=\eta z$ for some $\eta \in \C$. Now, using again that
$$v_1(z)= \log \frac{|(f \circ S)'(z)|}{1+\left| (f \circ
    S)(z)/z\right|^2}-\log |S'(z)|$$ is constant, we  see that $S'$ is constant
 and therefore $z_0=0$. Hence (\ref{eq:7}) is sharp only for $z_0=0$.
Now let $z_0=0$. Then  equality holds on either side of
(\ref{eq:7}) if and only if $f(z)=\eta z$ and 
$$ |\eta|=\frac{1 \pm \sqrt{1-4c^2}}{2c} \, .$$
If, in addition,  $c=1/2$, then equality holds in (\ref{eq:7}) on both sides,
so $|f'(0)|=1$, that is, every $f \in \mathcal{F}_c$ such that $f(0)=0$ has
the form $f(z)=\eta z$ with $|\eta|=1$.  This completes the proof.
  \end{proof}

The proof of Theorem \ref{thm:z01} is  based on the following simple  Schwarz--Pick type
lemma.

\begin{lemma} \label{lem:sp}
Suppose that $z_0 \in \D\setminus\{0\}$ and $w : \D \to \D$ is a holomorphic function such
that $w(z_0)=0$ and $w''(z_0)=0$. Then
$$ |w'(z_0)| \le  \frac{\sqrt{4+|z_0|^2}-|z_0|}{2 \left(1-|z_0|^2 \right)} \, .$$
Equality can hold only  if $w$ is a Blaschke product of
degree $2$.
\end{lemma}

In the proof we will identify all the extremal functions semi--explicitly. We intentionally
have excluded the case $z_0=0$ in Lemma \ref{lem:sp}.

\begin{proof}
  Write
  $$ w(z)=\frac{z-z_0}{1-\overline{z_0} z} \, g(z)$$
  for some holomorphic function $g : \D \to \overline{\D}$ and note that
  $g(\D) \subseteq \D$. Then $w''(z_0)=0$
  is equivalent to\begin{equation}
    \label{eq:5c}
    g'(z_0)=-\frac{\overline{z_0}}{1-|z_0|^2} \, g(z_0)  \, .
    \end{equation}
  The Schwarz--Pick lemma applied to $g$ implies
  $$ \frac{1-|g(z_0)|^2}{1-|z_0|^2} \ge |g'(z_0)|= \frac{|z_0|}{1-|z_0|^2} \,
  |g(z_0)| \, , $$
  which is equivalent to
  $$ |g(z_0)| \le \frac{\sqrt{4+|z_0|^2}-|z_0|}{2}\quad
    \Longleftrightarrow\quad |w'(z_0)| \le  \frac{\sqrt{4+|z_0|^2}-|z_0|}{2
      \left(1-|z_0|^2 \right)} \, .$$
  Again by the Schwarz--Pick lemma, we see that equality occurs if and only if
  $g$ is a unit disk automorphism such that (\ref{eq:5c}) holds.
\end{proof}

\begin{proof}[Proof of Theorem \ref{thm:z01}] In view of Theorem \ref{cor:2}
  we may assume that $z_0 \not=0$.
  We first closely follow  the proof
  of (\ref{eq:ba1}) in \cite{St} and  assume $f(z_0)=0$.
  Since $f \in \mathcal{M}(\D)$ is locally univalent its Schwarzian derivative
$$ S_f(z)= \left( \frac{f''(z)}{f'(z)} \right)'-\frac{1}{2} \left(
  \frac{f''(z)}{f'(z)} \right)^2$$
is \textit{holomorphic} in $\D$ and  we can
  write
  $$ f= \frac{w_1}{w_2} $$
  with holomorphic functions $w_1,w_2 : \D \to \C$ both of which  are 
  solutions of the linear second order ODE
  \begin{equation} \label{eq:7a}
    w''+\frac{S_f(z)}{2} w=0 \,.
  \end{equation}
  Note that the Wronskian $w_1' w_2-w_1 w_2'$ of $w_1$ and $w_2$ is constant, so by renormalizing the solutions we may assume that
  $$ w_1' w_2-w_1 w_2' \equiv 1 \, .$$
 In particular, we get
  $$ f^{\sharp}(z)=\frac{1}{|w_1(z)|^2+|w_2(z)|^2} \qquad \text{ and } 
  \qquad f^{\sharp}(z_0)=\frac{1}{|w_2(z_0)|^2}=|w_1'(z_0)|^2 \, .$$
  Since $f^{\sharp}(z) \ge c$, the first identity shows that $w:=\sqrt{c} \,
  w_1$ is a holomorphic selfmap of the unit disk  with $w(z_0)=0$ and
  $$w''(z_0)=-\frac{S_f(z_0)}{2}  w(z_0)=0\, ,$$
since $S_f$ is holomorphic. Hence we are in a position to apply  
Lemma \ref{lem:sp} and obtain
$$ f^{\sharp}(z_0) =|w_1'(z_0)|^2 =\frac{|w'(z_0)|^2}{c} \le \left( \frac{\sqrt{4+|z_0|^2}-|z_0|}{2} \right)^2 
    \frac{1}{c \left(1-|z_0|^2 \right)^2} \,. $$
  \end{proof}

  \begin{remark}
    Using the standard Schwarz--Pick lemma
    $$|w'(z_0)| \le  \frac{1}{1-|z_0|^2}$$
    instead of the ``improved'' Schwarz--Pick type Lemma
  \ref{lem:sp} in the last step of the preceding proof gives the less precise
  inequality  (\ref{eq:ba1}).  This is exactly the proof of (\ref{eq:ba1})
  given in \cite{St}. It does not fully use the fact that $S_f$ is
  \textit{holomorphic}. 
\end{remark}

\section{The Schwarz lemma for $\mathcal{F}_c$ and nonlinear boundary value problems}

In this section we show that the interior extremal problem solved in Theorem
\ref{cor:2} can be related to a ``dual'' nonlinear boundary extremal problem. 
This establishes a link between the Schwarz lemma for $\mathcal{F}_c$ and  a class of  boundary
value problems arising in conformal geometry which have first been investigated
by Beurling \cite{B}.

\medskip

The point of departure is the following  peculiar property of the
extremal functions in Theorem \ref{cor:2}:

\begin{theorem} \label{thm:ni}
  Let $c >0$ and $f \in \MD$ locally univalent. Then the following are equivalent:
  \begin{itemize}
  \item[(a)] $\displaystyle \lim \limits_{|z| \to 1} f^{\sharp}(z)=c$.
    \item[(b)]  $c \le 1/2$ and $f(z)=T(\eta z)$ with a
      rigid motion $T$ of the Riemann sphere and
  \begin{equation} \label{eq:eta} |\eta|= \frac{1 \pm \sqrt{1-4 c^2}}{2 c} \, .\end{equation}
          \end{itemize}
        \end{theorem}

In particular, this proves Theorem \ref{thm:kuh} which is merely a special case of
Theorem \ref{thm:ni}.

  \begin{proof} (b) $\Longrightarrow$ (a): This is just a computation.

    (a) $\Longrightarrow$ (b): This is a simple application of a rather deep
    result of Gidas, Ni and Nirenberg \cite{GNN}, which has become a standard
    tool in elliptic PDE, in combination with a nonlinear version of the
    Schwarz reflection principle, see \cite{R}. Let $f \in \MD$ be locally univalent and satisfy
    condition (a). By \cite[Theorem 1.8]{R}, we infer that $f$ has a
    meromorphic continuation to an open neighborhood of the closed unit disk
    $\overline{\D}$. This shows that
    $$ u(z):=\log f^{\sharp}(z)-\log c$$
    is a  $C^2$--function on $\overline{\D}$, that is, on an open neighborhood
    of $\overline{\D}$, such that
\begin{equation} \label{eq:bvp} \Delta u=-4 c^2 e^{2u}  \text{ on } \overline{\D} \quad \text{ and } \quad u=0 \text{ on }
  \partial \D \, .
\end{equation}
By the minimum principle, the superharmonic function $u$ is
\textit{positive} on $\D$. Hence
Theorem 1 in \cite{GNN} forces $u$ to be \textit{radially} symmetric,
$$ u(z)=v(r) \qquad (r=|z|)$$
for some strictly decreasing $C^2$--function $v : [0,1] \to [0,\infty)$.  It is now a
simple matter to see that all radially symmetric solutions of the
boundary value problem (\ref{eq:bvp})   have the form
\begin{equation} \label{eq:sol}
u(z)=\log \frac{|\eta|}{1+|\eta|^2 |z|^2}-\log c
\end{equation}
with $\eta \in \C$ as in (\ref{eq:eta}). For convenience, we indicate the main
steps. 
Since
$$ \Delta u(z)= \frac{1}{r} \left( r v'(r) \right)' \, ,$$
where $r=|z|$ and $'$ indicates differentiation with respect to $r$, 
the strictly decreasing  function  $v \in C^2([0,1])$ solves
$$ (r v'(r))'= -4 c^2 r \, e^{2 v(r)} \text{ on } [0,1] \, , \qquad v(1)=0 \, .$$
We substitute $r=e^{x}$ and obtain for $w(x):=v(e^x)+x+\log (2c)$ the initial
value problem
$$ w''(x)=  - e^{2 w(x)} \text{ on } (-\infty,0] \, , \qquad
w(0)=\log(2c) \, .$$
This ODE has $2 \, w'(x)$ as an integrating factor,  so
$$ \left( w'(x)^2 \right)'= - \left( e^{2 w(x)} \right)' \, .$$ 
Integrating from $x=a$ to $x=t$ and using  that $$\lim \limits_{a \to -\infty}
w'(a)=\lim \limits_{a \to -\infty}e^{a} v'(e^a)+1= 1$$ as well as
$$ \lim \limits_{a \to -\infty} w(a) = -\infty\, ,$$
we arrive at
\begin{equation} \label{eq:ODE}
  w'(x)^2=1-e^{2 w(x)} \quad \text{ on } (-\infty,0] \, .
\end{equation}

In particular, $w(x) \le 0$ for all $x \in (-\infty,0]$. This implies $\log (2c)=w(0) \le 0$, and hence $c \le 1/2$.
Note also that there exists no subinterval $[a,b] \subseteq (-\infty,0]$ such that $w(a)=w(b)=0$ and $w(x) \not=0$ for all $x \in (a,b)$, since otherwise we would have $w'(c)=0$ for some $c \in (a,b)$ and so $w(c)=0$ by (\ref{eq:ODE}).
Since $w''=-e^{2 w}$, the $C^2$--function $w$ cannot vanish identically on any
open interval, so it follows that $w$ vanishes at one point at most. 
Therefore,
$$ w'(x)=\pm \sqrt{1- e^{2 \, w(x)}} \, $$ 
with at most one change of sign on $(-\infty,0]$.
The resulting two ODEs (one for each sign) are both separable. Hence if one is
patient enough,  elementary integration on each interval where $w\not=0$, would
ultimately show that the solutions $u(z)$ to (\ref{eq:bvp}) have the form
(\ref{eq:sol}). However, one can avoid this lengthy calculation as
follows. Let us first consider the case $c<1/2$. If $w(x)$
has no zero in $(-\infty,0]$, then $w<0$ there and  in view of $w'(x) \to 1$ as $x \to -\infty$ we
see that 
$$ w'=\sqrt{1-e^{2w}} \, , \qquad w(0)=\log (2c) < 0 \, , $$
Hence $w$ is uniquely determined by the standard uniqueness result for ODEs.
If $w$ has exactly one zero $\alpha \in (-\infty,0)$, then $w$ is a solution of  the initial
value problem
$$ w'=-\sqrt{1-e^{2w}} \, , \qquad w(0)=\log (2c) < 0 \, ,$$
on $[\alpha,0]$, which has a unique solution on $(\alpha,0]$. Hence $\alpha$
and $ w : [\alpha,0] \to \R$ are uniquely determined by $c$. Since
$w(\alpha)=0$, we have $w'(\alpha)=0$ by $(\ref{eq:ODE})$. 
Thus $w$ solves the initial value problem $w''=-e^{2w}$,
$w(\alpha)=w'(\alpha)=0$, and is therefore uniquely determined also on $(-\infty,\alpha]$.
In the  remaining  case $w(0)=0$ and $c=1/2$, we have $w'(0)=0$ and
$w''=-e^{2w}$, so $w$ is uniquely determined as before.
We
therefore see that there are at most two solutions to (\ref{eq:bvp}) which
coincide for $c=1/2$. This shows that there are no solutions  $u(z)$ to
(\ref{eq:bvp}) different from those given by  (\ref{eq:sol}) with $\eta \in \C$ as in
(\ref{eq:eta}).  
\hide{only have to consider the
initial valu eproblem 
This leads to two explicit global solutions $w : (-\infty,0] \to \R$ (which coincide for $c=1/2$)  and ultimately shows that  the solutions $u(z)$ to (\ref{eq:bvp}) have the form (\ref{eq:sol}).}
    \end{proof}

    \begin{theorem}[The Schwarz lemma for $\mathcal{F}_c$ and a dual boundary extremal problem] \label{cor:ni}
Let $c>0$ and $F \in \mathcal{F}_c$. Then the following are equivalent.
\begin{itemize}
\item[(a)] $F$ is extremal for one of the interior extremal problems $$\max \limits_{f \in \mathcal{F}_c}
  f^{\sharp}(0) \quad \text{ or } \quad \min \limits_{f \in \mathcal{F}_c}
  f^{\sharp}(0) \, . $$
  \item[(b)] $F$ is extremal for the boundary extremal problem $$\min \limits_{f \in \mathcal{F}_c}
    \limsup \limits_{z \to \xi} f^{\sharp}(z)\quad \text{ for every } \xi \in
    \partial \D \, .$$
  \end{itemize}
\end{theorem}

The proof of Theorem \ref{cor:ni} is by now obvious because we have identified all functions $F \in \mathcal{F}_c$
with property (a) in Theorem \ref{cor:2} and those with property (b) in
Theorem \ref{thm:ni} in an explicit way.
It would be desirable to have a direct proof of the fact
that (a) and (b) are equivalent.

\begin{problem}
Theorem  \ref{cor:ni} roughly says that every $f \in \mathcal{F}_c$ that
maximizes/minimizes $f^{\sharp}$ at the \textit{origin} actually minimizes
$f^{\sharp}$ on the \textit{entire} unit circle. Now suppose that $f \in
\mathcal{F}_c$ maximizes/minimizes $f^{\sharp}$ over the set $\mathcal{F}_c$ at a point $z_0\not=0$. Does 
$f^{\sharp}$ have a corresponding boundary extremal property on (part of) the unit circle\,?
\end{problem}

We are now in a position to 
 relate the Schwarz lemma for the class
 $\mathcal{F}_c$  with Beurling's well--known extension
 of the Riemann mapping theorem (see \cite{Av94, Av96, B,
   C2, BKRW, FRold}).  
 Denote by $\mathcal{H}_0(\D)$ the set of all holomorphic functions $g : \D \to \C$ with
 $g(0)=0$ and $g'(0)>0$. For a given  positive, continuous and \textit{bounded} function
 $\Phi : \C \to \R$,  Beurling \cite{B} considered the
 nonlinear boundary value problem\footnote{This is a ``Riemann--Hilbert--Poincar\'e problem''.}
\begin{equation} \label{eq:beurling}
  \lim \limits_{|z| \to 1} \big( |g'(z)|-\Phi(g(z)) \big)=0
  \end{equation}
 and showed that this problem always admits \textit{univalent} solutions $g
 \in \mathcal{H}_0(\D)$. In fact, Beurling even showed that there is always 
a kind of  ``maximal'' resp.~``minimal'' univalent solution. In order to find the
 ``minimal'' univalent solution,  Beurling considered the set of univalent
 ``supersolutions'' of (\ref{eq:beurling}), 
 $$ B_{\Phi}:=\left\{g \in \mathcal{H}_0(\D) \text{ univalent} \, \Big| \, 
       \liminf  \limits_{|z| \to 1} \big( |g'(z)|-\Phi(g(z)) \big) \ge 0 \right\}\, ,$$ 
     and proved in a first step that there is a unique function $g^* \in B_{\Phi}$ such that
     $$ {g^*}'(0)=\inf \limits_{g \in B_{\Phi}} g'(0)\, .$$
     In a second step, he then showed that
     this ``minimal'' supersolution is in fact a solution of the boundary
     value problem (\ref{eq:beurling}). It appears that for Beurling's method the assumption that $\Phi$ is
     \textit{bounded} (or at least of sublinear growth as in \cite{Av94}) is fairly essential.
     Now, it is easy to see that Beurling's set  of supersolutions for the \textit{unbounded}
     function
     $$ \Phi_c(w):=c \left(1+|w|^2 \right) $$
     can be written as
     $$ B_{\Phi_c}=\left\{g \in \mathcal{H}_0(\D) \text{ univalent} \, \Big|
       \, 
       \liminf  \limits_{|z| \to 1} \frac{|g'(z)|}{1+|g(z)|^2} \ge c
     \right\} \, .$$
Since $\log g^{\sharp}$ is superharmonic for every (locally) univalent function $g$, we hence see that
$$ B_{\Phi_c} \subseteq \mathcal{F}_c \, .$$
Therefore, Theorem \ref{cor:2} implies that for any $c \in (0,1/2]$
the function
$$ g_c(z) =\frac{1-\sqrt{1-4c^2}}{2 c} z \, ,$$
which belongs to $ B_{\Phi_c}$, is the unique extremal function for the extremal problem
$$ \inf \limits_{g \in B_{\Phi_c}} g'(0)  \, , $$
 and $g_c$ 
is obviously a   solution to Beurling's boundary value problem (\ref{eq:beurling})
for $\Phi=\Phi_c$. Clearly, an analogous result holds for the unique function
in $B_{\Phi_c}$ which maximizes $g'(0)$ for all $g \in B_{\Phi_c}$. By Theorem
\ref{thm:ni} and the superharmonicity of $\log g^{\sharp}$ for any locally univalent function $g \in \mathcal{H}_0(\D)$, these two solutions are the only two \textit{locally univalent} solutions to
(\ref{eq:beurling}) for $\Phi=\Phi_c$\,!
To put it differently, Theorem \ref{cor:2} and Theorem
\ref{thm:ni} provide an extension of  Beurling's results at least for the
specific function
$$ \Phi(w)=\Phi_c(w)=c \left(1+|w|^2 \right) \, , \qquad c \le 1/2 \, ,$$
which is of \textit{quadratic} and not merely sublinear growth. 

\medskip

For convenience, we state these considerations as a theorem, which as we have
seen is now merely a restatement of Theorem \ref{cor:2} and  Theorem \ref{thm:ni}.

   \begin{theorem}[The Beurling--Riemann mapping theorem for the spherical
     metric] \label{thm:beu}
     Suppose that $c>0$ and consider the boundary value problem
    \begin{equation} \label{eq:be}
 \lim \limits_{|z| \to 1}  \big( |g'(z)|-c \left(1+|g(z)|^2 \right)
  \big)=0 
\end{equation}
for $g \in \mathcal{H}_0(\D)$.
     \begin{itemize}
     \item[(a)] If $c <1/2$, then (\ref{eq:be}) has exactly two locally
       univalent solutions $g_{\pm} \in \mathcal{H}_0(\D)$ given by
       $$ g_{\pm}(z)=\frac{1\pm\sqrt{1-4 c^2}}{2c} z \, .$$
       These solutions are univalent and they are the  unique extremal functions for the
       extremal problems 
              $$ \max \{ g'(0) \, : \, g \in B_{\Phi_c}\}  \quad \text{ and } \quad
       \min \{ g'(0) \, : \, g \in B_{\Phi_c}\} \, .$$
       
     \item[(b)] If $c=1/2$, then $g(z)=z$ is the only locally univalent
       solution  $g \in
       \mathcal{H}_0(\D)$ of (\ref{eq:be}).
         \item[(c)] If $c>1/2$, then (\ref{eq:be}) has no locally univalent
           solution in $\mathcal{H}_0(\D)$.
       \end{itemize}
     \end{theorem}

We note that parts of Theorem \ref{thm:beu} have been proved earlier by
different means, see e.g.~\cite{Av94,Av96,FRold,K}. The essential new ingredient is
the uniqueness statement in part (a), which ultimately comes from the
Gidas--Ni--Nirenberg theorem. For this  uniqueness result  the local univalence assumption is necessary.
In fact, K\"uhnau \cite[Satz 4]{K} has shown that there are always additional solutions $g \in \mathcal{H}_0(\D)$ for Beurling's boundary value problem (\ref{eq:be}) which, however, are  not locally univalent.

\section{Concluding remarks}

\begin{remark} \label{rem:final}
In all of our results, the restriction to \textit{locally univalent} functions
is essential. The reason is that  $\log f^{\sharp}$ is
   superharmonic only for  \textit{locally univalent} functions $f \in \MD$,
   so  the  minimum principle can be applied and  shows that  
   $$ \mathcal{F}_c=\left\{ f \in \MD \text{ locally univalent} \, : \,
     \liminf \limits_{|z| \to 1} f^{\sharp}(z) \ge c \right\}$$
In fact, the larger class
$$ \mathcal{G}_c:=\left\{ f \in \MD \, : \,
     \liminf \limits_{|z| \to 1} f^{\sharp}(z) \ge c \right\}$$
   is not even a normal family in view of the following example.
\end{remark}
   
\begin{example}
 Let $$ g_n(z):=
\frac{z}{1/n^2 + z^2}\, .$$
Clearly ${ g_n} \in  \MD$ and a straightforward computation leads to 
$$  g_{n}^{\sharp}(z) =  \frac{ |1/n^2 -  z^2|}{|1/n^2 + z^2|^2 + |z|^2} \ge \frac{1-1/n^2}{2+2/n^2+1/n^4}$$
and hence $g_n \in \mathcal{G}_c$ for any $c \in (0,1/2)$ for all but finitely
many $n$. 
 However it is readily checked that 
$ g_n(0) = 0$ but $$\lim\limits_{n\rightarrow \infty}z  g_n(z)  = 1$$ on the
punctured unit disk $\D\setminus\{0\}$. Hence none of the families
$\mathcal{G}_c$, $c \in  (0,1/2)$, is a normal family. 
\end{example}

\begin{remark}
There is another simple proof of the Grahl--Nevo normality criterion. Let us set 
$$\mathcal{F}_{c,0}:=\left\{f \in \mathcal{F}_c \,:  \,f(0) = 0\right\}\,.$$ 
It is clear that for any $ f \in  \mathcal{F}_c $ and any $z \in  \D ,$
\begin{equation} \label{eq:8}
  |f'(z)| \ge f^{\sharp}(z) \ge c ,
\end{equation}
and by the fundamental normality test \cite[p.~74]{Sc} the family consisting of
derivatives of functions  $ f \in  \mathcal{F}_c $ is a normal family; note
that this also follows from the plain fact that the family    $$ \left\{ {1\over f'} \, : \, f \in \mathcal{F}_c \right\} $$
 contains only functions
analytic in the unit disk and is also uniformly bounded above there by
(\ref{eq:8}). Now since for any $f$ in $\mathcal{F}_{c,0}$ we have
$$\int\limits_{0}^{z} f'(t) dt = f(z)\, ,$$ we obtain (see e.g.~Lemma 8 in \cite{GN}) that $\mathcal{F}_{c,0}$ is a normal family. We define for each 
$f \in \mathcal{F}_c $, $$ F(z) := \begin{cases} \displaystyle
  \frac{f(z)-f(0)}{1+\overline{f(0)}f(z)} & \text{ if } f(0) \in \C \\[6mm]
  \displaystyle \frac{1}{f(z)} & \text{ otherwise\,,}
  \end{cases}$$
  so $F$ belongs to $\mathcal{F}_{c,0}$. Now
let  $\{f_n \}$ be a sequence in $ \mathcal{F}_c $. Then  $F_n \in
\mathcal{F}_{c,0} $ for each $n$ 
and therefore  $\{F_{n}\}$  has a subsequence $\{F_{n_k}\}$ which converges uniformly on compact subsets of $\D$. If
$\{f_{n_k}(0)\}$ does not converge to the point at infinity, then  passing to
a further subsequence if necessary, we may assume that $\{f_{n_k}(0)\}
\subseteq \C$ and that $f_{n_k}(0) \to \alpha \in \C$. Hence the sequence 
$\{f_{n_k}\}$ is compactly convergent in $\D$  because  
$$  f_{n_k}(z) =\frac{F_{n_k}(z)+f_{n_k}(0)}{1-\overline{f_{n_k}(0)}F_{n_k}(z)}. $$
A similar line of reasoning is available if $\{f_{n_k}(0)\}$ does converge to
the point at infinity and we may therefore conclude that $ \mathcal{F}_c $ is a normal family.
\end{remark}

\begin{remark}
It is sometimes possible to give straightforward proofs of the normality of specific subclasses of $ \mathcal{F}_c $. Let $\mathcal{P}_n$ denote the class of complex polynomials of degree at most $n$ and  $\mathcal{R}_n$ the class of rational functions $f=p_n/q_n$ with $p_n \in \mathcal{P}_n$ and
 $$ q_n(z) = \prod \limits_{j=1}^{n} (z-z_{j})$$ where the points ${z_{j}}$ are
 fixed once for all with $|z_{j}|>1$ . We set $$||f||= \sup \limits_{z \in
   \partial \D}
 |f(z)|\, .$$ According to an estimate of Borwein and Erdelyi \cite{BE},
$$ |f'(z)| \le K(z)||f|| \quad \text{ for any }  |z|=1$$
with
$$ K(z)= \sum_{j=1}^{n} \frac{|z_{j}|^2 -1 }{|z_{j}-z|^2 }$$
and clearly  $$ K(z)\le \sum_{j=1}^{n} \frac{|z_{j}|+1 }{|z_{j}|-1 }:=k_n\,
.$$ In particular, if $ f\in  \mathcal{F}_c \cap \mathcal{R}_n $ and if $
|f(z_0)|=||f||$ for some $|z_0|=1$, then
$$ c \le \frac{|f'(z_0)|}{1+|f(z_0)|^2 } \le k_n \frac{|f(z_0)|}{1+|f(z_0)|^2 } .$$
It follows that $c/k_n\le 1/2$ and
$$ ||f||=|f(z_0)| \le \frac{k_n}{2c} \left(1+\sqrt{ 1-\frac{4c^2}{k_n^2}} \right)\,.$$ 
The family $\mathcal{F}_c \cap \mathcal{R}_n $ is therefore uniformly bounded on the unit disk and in particular normal there.
\end{remark}

{The authors would like to express their gratitude to an anonymous
     referee for his or her careful reading of this paper and many valuable suggestions.}


\bibliographystyle{amsplain}

\vfill

Richard Fournier\\
Department of Mathematics\\
University of Montreal\\
CP 6128, Succursale Centre-Ville\\
Montreal\\
H3C3J7 Canada\\
fournier@dms.umontreal.ca\\

Daniela Kraus \\
Department of Mathematics\\
University of W\"urzburg\\
Emil Fischer Stra{\ss}e 40\\
97074 W\"urzburg\\
Germany\\
dakraus@mathematik.uni-wuerzburg.de\\

 Oliver Roth\\
Department of Mathematics\\
University of W\"urzburg\\
Emil Fischer Stra{\ss}e 40\\
97074 W\"urzburg\\
Germany\\
roth@mathematik.uni-wuerzburg.de\\

\end{document}